\documentclass[letterpaper, 10 pt, conference]{ieeeconf}  

\IEEEoverridecommandlockouts                              
\usepackage{additionalSetting}
\usepackage{anyfontsize}



\title{\LARGE \bf
A structure exploiting SDP solver for robust controller synthesis
}

\author{Dennis Gramlich, Tobias Holicki, Carsten W. Scherer and Christian Ebenbauer
\thanks{Funded by Deutsche Forschungsgemeinschaft (DFG, German Research Foundation) under Germany's Excellence Strategy - EXC 2075 – 390740016. We acknowledge the support by the Stuttgart Center for Simulation Science (SimTech).}
\thanks{Dennis Gramlich and Christian Ebenbauer are with the Chair of Intelligent Control Systems,
        RWTH Aachen University,
        52074 Aachen, Germany
        {\tt\small \{dennis.gramlich,christian.ebenbauer\} @ic.rwth-aachen.de}}%
\thanks{Tobias Holicki and Carsten W. Scherer are with the Chair of Mathematical Systems Theory, University of Stuttgart, 70569 Stuttgart, Germany {\tt\small \{tobias.holicki,carsten.scherer\} @imng.uni-stuttgart.de}}%
}%

\begin{document}

\maketitle
\thispagestyle{empty}
\pagestyle{empty}

\begin{abstract}

In this paper, we revisit structure exploiting SDP solvers dedicated to the solution of Kalman-Yakubovic-Popov semi-definite programs (KYP-SDPs). These SDPs inherit their name from the KYP Lemma and they play a crucial role in e.g. robustness analysis, robust state feedback synthesis, and robust estimator synthesis for uncertain dynamical systems. Off-the-shelve SDP solvers require $O(n^6)$ arithmetic operations per Newton step to solve this class of problems, where $n$ is the state dimension of the dynamical system under consideration. Specialized solvers reduce this complexity to $O(n^3)$. However, existing specialized solvers do not include semi-definite constraints on the Lyapunov matrix, which is necessary for controller synthesis. In this paper, we show how to include such constraints in structure exploiting KYP-SDP solvers.

\end{abstract}

\section{INTRODUCTION}
\label{sec:1}

Let $\bbS^n$ denote the set of symmetric matrices of dimension $n$. In this work, we study optimization problems with semi-definite constraints of the form
\begin{subequations}
\label{eq:coreProblem}
\begin{align}
	&\minimize_{\lambda \in \bbR^p, P \in \bbS^n} ~~ c^\top \lambda - \trace (\Sigma P) \label{eq:cost}\\
	&\mathrm{s.t.} ~~
	\begin{pmatrix}
		A & B\\
		I & 0
	\end{pmatrix}^\top 
	\begin{pmatrix}
		0 & P\\
		P & 0
	\end{pmatrix}
	\begin{pmatrix}
		A & B\\
		I & 0
	\end{pmatrix}
	+
	\begin{pmatrix}
		Q(\lambda) & S(\lambda)\\
		S^\top (\lambda) & R(\lambda)
	\end{pmatrix}
	\prec 0, \label{eq:LMI}\\
	&\hspace{8mm} N(\lambda) \succ 0,\label{eq:MultiplierConstraint}\\
    &\hspace{8mm} P \succ 0, \label{eq:PositivityConstraint}
\end{align}
\end{subequations}
where $H(\lambda) := H_0 + \sum_{i = 1}^p \lambda_i H_i$ for $H \in \{N, Q, S, R\}$ are affine matrix valued functions of $\lambda$. The matrix parameters are chosen to be of compatible dimensions, i.e., $Q_i \in \bbS^n$, $S_i \in \bbR^{n\times m}$, $R_i \in \bbS^m$, $N_i \in \bbS^r$, $A \in \bbR^{n\times n}$, $B \in \bbR^{n\times m}$ and $c \in \bbR^p$ for $i = 0,\ldots, p$ and positive integers $n,m,p,r \in \bbN$. 
Finally, we assume that the matrix pair $(A,B)$ is controllable and that the matrix $\Sigma \in \bbS^n$ is positive semi-definite. Linear matrix inequalities of the form \eqref{eq:LMI} frequently appear in control and signal processing and are related to the celebrated KYP-Lemma \cite{Ran96}. A partial list of applications for the SDP \eqref{eq:coreProblem} includes robustness analysis, robust controller synthesis, and robust estimator synthesis for linear dynamical systems. Today, the solution of online SDPs, e.g., for online data-driven control, or the solution of large SDPs, e.g., for estimating the Lipschitz constant of recurrent neural networks, should be of particular interest.

In many scenarios where the SDP \eqref{eq:coreProblem} appears, the state dimension $n$ is larger than the number $p$ of multipliers $\lambda_i$. In this case, the $O(n^2)$ decision variables in the matrix variable $P$ typically dominate the computational effort for solving \eqref{eq:coreProblem}. With off-the-shelve barrier function methods, for example, the cost for one Newton step scales with $n^6$. The complexity in $n$ can be improved by exploiting the structure of KYP-SDPs. This is explored in \cite{balakrishnan1999efficient,parrilo2000structured,kao2001fast,hansson2001primal} for \eqref{eq:cost}-\eqref{eq:MultiplierConstraint} but without \eqref{eq:PositivityConstraint}. A notable example is \cite{wallin2004kypd,VanBal05}, which cleverly parameterizes the variables of the dual of \eqref{eq:cost}-\eqref{eq:MultiplierConstraint} and thereby reduces the cost for one Newton iteration to $O(n^4)$.

Further structure exploiting algorithms for \eqref{eq:cost}-\eqref{eq:MultiplierConstraint} include cutting plane methods \cite{wallin2005decomposition,wallin2008cutting,falkeborn2012decomposition,JaoPal19,abou2019oracle}. These methods optimize over $P$ in an inner loop, whereas cutting planes for $\lambda$ are constructed in an outer loop. Such a splitting approach enables a more efficient optimization over $P$, e.g., by solving Riccati equations. Consequently, cutting plane methods are effective when the number of variables $\lambda_i$ is very small, but according to \cite{KaoMeg07} probably less effective when this number is moderate.

Alternatively to optimizing over $P$ in \eqref{eq:coreProblem}, one can also approach \eqref{eq:cost}-\eqref{eq:MultiplierConstraint} by solving the equivalent frequency domain inequality. This is considered in \cite{LiuVan07}, where the frequency domain inequality, which involves an infinite number of semi-definite constraints, is solved using a sampling approach. The latter can reduce the computational effort for Newton iterations, but produces only a lower bound on the optimal value. For this reason, in \cite{KaoMeg07}, a barrier function for the frequency domain inequality over all frequencies is constructed. Evaluating this barrier function requires solving Riccati and Lyapunov equations in each inner loop iteration. In addition, \cite{KaoMeg07} differentiates through the Riccati and Lyapunov equation to enable efficient optimization also for moderate numbers of $\lambda_i$ using second-order optimization algorithms.

In the present paper, we extend the problem \eqref{eq:cost}-\eqref{eq:MultiplierConstraint} studied in the cited KYP-SDP literature, with the semi-definite constraint \eqref{eq:PositivityConstraint}. This constraint enables, for example, robust state-feedback synthesis. Methodologically, we employ a second-order optimization algorithm to minimize a barrier function relaxation of \eqref{eq:coreProblem} similarly to \cite{KaoMeg07}. To this end, our key step is introducing a convex barrier function for the existence of a solution to a Riccati equation.


\section{Problem statement}
\label{sec:2}

Since the problem \eqref{eq:coreProblem} can be expensive to solve by off-the-shelve SDP solvers, we study the alternative problem
\begin{align}
	\minimize_{\lambda \in \bbR^p} ~&~ c^\top \lambda - \trace \Sigma P_+(\lambda) \label{eq:alternativeProblem}\\
	\mathrm{s.t.} ~&~ N(\lambda) \succ 0, P_+(\lambda) \succ 0, 
    \lambda \in \calD. \nonumber
\end{align}
Here, the function $\lambda \mapsto P_+(\lambda)$ is defined as the anti-stabilizing solution $P$ of the Riccati equation $F(P,\lambda) = 0$, where $F$ is defined as
\begin{align}
	F(P,\lambda) = A^\top P + P A + Q - (PB + S) R^{-1} (PB + S)^\top
	\label{eq:Riccati}
\end{align}
and where we abbreviate $Q = Q(\lambda)$, $S = S(\lambda)$ and $R = R(\lambda)$. Further, $\calD$ is defined as the set of all $\lambda \in \bbR^p$ with $R(\lambda) \prec 0$ for which $F(\cdot, \lambda) = 0$ has an anti-stabilizing solution. This problem formulation is motivated by the following extended version of the KYP-Lemma \cite{scherer2020theory}.

\begin{lemma}
    \label{lem:KYP}
	Consider a fixed $\lambda \in \bbR^p$ and suppose that $\mathrm{eig}(A) \cap i\mathbb{R} = \emptyset$ and that $(A, B)$ is controllable. Then the following statements are equivalent.
	\begin{enumerate}
		\item $\begin{pmatrix}
		    (A - i\omega I)^{-1} B\\
                I
		\end{pmatrix}^*
            \begin{pmatrix}
		Q & S\\
		S^\top & R
	\end{pmatrix}
        \begin{pmatrix}
		    (A - i\omega I)^{-1} B\\
                I
		\end{pmatrix} \prec 0$ for all $\omega \in \bbR \cup \{\infty\} . 
        $
		\item There exist some symmetric $P$ satisfying \eqref{eq:LMI}.
		\item $R(\lambda)\prec 0$ and there exists $P \in \bbS^n$ with $F(P,\lambda) \prec 0$.
		\item $R(\lambda)\prec 0$ and the Riccati equation  $F(P,\lambda ) = 0$ has an anti-stabilizing solution $P_+(\lambda) \in \bbS^n$.
	\end{enumerate}
\end{lemma}

We mention that \cite{KaoMeg07,LiuVan07} are based on 1), whereas we utilize 4).

In the course of our exposition, we show that our formulation \eqref{eq:alternativeProblem} provides the same numerical advantages as \cite{KaoMeg07}, but additionally allows us to consider the constraint \eqref{eq:PositivityConstraint} and the cost term $- \trace \Sigma P$. The challenge is to handle the constraint $\lambda \in \calD$, i.e., the feasibility of the Riccati equation, and the nonlinear function $\lambda \mapsto P_+(\lambda)$. We address these challenges by deriving a convex barrier function for the feasibility of the Riccati equation and by showing that the mapping $\lambda \mapsto P_+(\lambda)$ is concave (in the sense of Hermitian valued functions).

We conclude the section with an equivalence theorem for \eqref{eq:coreProblem} and \eqref{eq:alternativeProblem} which is proven in  Section \ref{sec:4}.

\begin{theorem}[Equivalence of \eqref{eq:coreProblem} and \eqref{eq:alternativeProblem}]
	\label{thm:equivalence}
	Problem \eqref{eq:coreProblem} and problem \eqref{eq:alternativeProblem} are equivalent, i.e., the optimal values coincide and $\lambda \in \bbR^p$ is feasible for \eqref{eq:alternativeProblem} if and only if there exists $P \in \bbS^n$ such that $(\lambda,P)$ is feasible for \eqref{eq:coreProblem}.
 \end{theorem}

\section{An interior point method for \eqref{eq:alternativeProblem}}


To solve \eqref{eq:alternativeProblem}, we propose the employment of a path-following barrier method similar to \cite{KaoMeg07}. For this purpose, a barrier function for the constraint $\lambda \in \calD$ is given by $\lambda \mapsto -\log\det (-R(\lambda)) -\log \det \Delta (\lambda)$ where $\Delta (\lambda) := P_+ (\lambda) - P_- (\lambda)$ is the difference between the stabilizing solution $P_-(\lambda)$ and the anti-stabilizing solution $P_+(\lambda)$ of the Riccati equation. This fact is proven in Section \ref{sec:4}.
For the remaining semi-definite constraints, we utilize the standard $\log \det$ barrier function. Overall, for an increasing sequence of $t$, we minimize
\begin{align}
    v_t(\lambda) &=t(c^\top \lambda - \trace \Sigma P_+(\lambda)) - \log \det N(\lambda)  \label{eq:v_t}\\
    &- \log \det P_+(\lambda)
    - \log \det (- R(\lambda)) - \log \det \Delta(\lambda)  \nonumber
\end{align}
as a function of $\lambda$.
To solve this optimization problem, we need to determine first- and second-order derivatives of the solutions $P_+(\lambda)$, $P_-(\lambda)$ of the Riccati equation $F(P,\lambda) = 0$. To simplify the notation we drop the argument $\lambda$ in our matrix-valued functions sometimes.

\begin{theorem}
	\label{thm:riccatiDifferentiation}	
 
    Given $\lambda_0 \in \bbR^p$ and $P_0 \in \bbS^n$ with $F(P_0,\lambda_0) = 0$, if $A - BK$ has no eigenvalues on the imaginary axis, where $K := R^{-1}(P_0B + S)^\top$, then there exist a neighbourhood $\calN$ of $\lambda_0$ and an arbitrarily often differentiable function $P: \calN \to \bbS^n$ with $P(\lambda_0) = P_0$, such that $F(P(\lambda),\lambda) = 0$ for all $\lambda \in \calN$. Moreover, the partial derivative $\partial_{\lambda_i} P$ is the solution of the Lyapunov equation
	\begin{align}
			0 = \partial_{\lambda_i}P (A - B K) &+ (A - B K)^\top \partial_{\lambda_i} P \nonumber\\
			& +
			\begin{pmatrix}
				I\\
				-K
			\end{pmatrix}^\top
			\begin{pmatrix}
				Q_i & S_i\\
				S_i^\top & R_i
			\end{pmatrix}
			\begin{pmatrix}
				I\\
				-K
			\end{pmatrix}.\label{eq:RiccatiFirstOrder}
	\end{align}
	Furthermore, the second order partial derivative $\partial_{\lambda_i} \partial_{\lambda_j} P$ is the unique solution of the Lyapunov equation
	\begin{align}
		0 = (A - B K)^\top \partial_{\lambda_i} \partial_{\lambda_j} P + \partial_{\lambda_i} \partial_{\lambda_j}P (A - B K)\nonumber\\
		- \partial_{\lambda_j} K^\top R \partial_{\lambda_i} K - \partial_{\lambda_i} K^\top R \partial_{\lambda_j} K, \label{eq:RiccatiSecondOrder}
	\end{align}
	where $\partial_{\lambda_i} K := R^{-1}(B^\top \partial_{\lambda_i} P + S_i^\top - R_i K)$.
\end{theorem}

For the proof, we refer to \cite{KaoMeg07,curtain2010analytic}.

Theorem \ref{thm:riccatiDifferentiation} enables us to differentiate the solutions $P_+(\cdot)$ and $P_-(\cdot)$ of the Riccati equation. As a consequence, we can formulate the path-following interior point method Algorithm \ref{alg:1} for solving \eqref{eq:alternativeProblem}. Derivatives of the barrier functions are derived using standard formulas and provided in Algorithm \ref{alg:1}.

\begin{algorithm}
	\caption{Solver for \eqref{eq:alternativeProblem}}\label{alg:1}
	\begin{algorithmic}
		\STATE \textbf{Input:} $\varepsilon,t_{\max}$, initial feasible point $\lambda$ of \eqref{eq:alternativeProblem}.
		\WHILE{$t \leq t_{\max}$}
		\STATE $P_- \gets$ stabilizing solution of $F(P,\lambda) = 0$
		\STATE $P_+ \gets$ anti-stabilizing solution of $F(P,\lambda) = 0$
            \STATE $(Q,S,R,N) \gets (Q(\lambda),S(\lambda),R(\lambda),N(\lambda))$
		\STATE $v_t \gets - \log \det P_+ - \log \det (-R) - \log \det N$
            \STATE ~~~~~ $- \log \det \Delta  + t(c^\top \lambda - \trace \Sigma P_+)$
            \STATE $(\nabla v_t)_i \gets -\trace \Delta^{-1} \partial_{\lambda_i} \Delta + t (c_i - \trace \Sigma \partial_{\lambda_i} P_+)$\\
            \STATE ~~~~~ $-\trace P_+^{-1} \partial_{\lambda_i} P_+ -\trace R^{-1}R_i - \trace N^{-1}N_i$
            \STATE $(H_{v_t})_{ij} \gets \trace P_+^{-1} (2\partial_{\lambda_i} P_+ P_+^{-1} \partial_{\lambda_j} P_+ - \partial_{\lambda_i} \partial_{\lambda_j} P_+)$
            \STATE ~~~ $ + 2\trace R_i R^{-1} R_j R^{-1} + 2\trace N_i  N^{-1} N_j N^{-1}$
            \STATE ~~~ $ + \trace \Delta^{-1}(2\partial_{\lambda_i} \Delta \Delta^{-1} \partial_{\lambda_j} \Delta - \partial_{\lambda_i} \partial_{\lambda_j} \Delta)$
            \STATE ~~~ $ - t \trace \Sigma \partial_{\lambda_i} \partial_{\lambda_j} P_+$
            \STATE $d \gets -H_{v_t}^{-1} \nabla v_t$ (Newton search direction)
		\STATE $\alpha \gets $ line search for $\argmin_{\alpha} v_t(\lambda + \alpha d)$
		\STATE $\lambda \gets \lambda + \alpha d$
            \STATE \textbf{if} (stopping criterion) \textbf{then} $t \gets 10t$
		\ENDWHILE
		\RETURN $(P_+,\lambda)$
	\end{algorithmic}
\end{algorithm}

\begin{remark}[Initial feasible points]
    \label{rem:initalFeasible}
    To generate an initial interior point for Algorithm \ref{alg:1}, we apply a standard procedure found in \cite{boyd2004convex} and extend our decision variable to $\tilde{\lambda} := \begin{pmatrix}
	\lambda_0 & \lambda_1 & \cdots & \lambda_p
\end{pmatrix}^\top$, the multiplier matrix to
\begin{align*}
	\begin{pmatrix}
		\widetilde{Q}(\tilde{\lambda}) & \widetilde{S}(\tilde{\lambda})\\
		\widetilde{S}(\tilde{\lambda})^\top & \widetilde{R}(\tilde{\lambda})
	\end{pmatrix}
	:=
	\begin{pmatrix}
		Q(\lambda) & S(\lambda)\\
		S(\lambda)^\top & R(\lambda)
	\end{pmatrix}
	-
	\lambda_0
	\begin{pmatrix}
		I & 0\\
		0 & I
	\end{pmatrix},
\end{align*}
and $N$ to $\widetilde{N}(\tilde{\lambda}) := N(\lambda) + \lambda_0 I$. Then $\tilde{\lambda}$ is an interior point for the modified problem \eqref{eq:coreProblem} with $\widetilde{Q}(\cdot), \widetilde{R}(\cdot), \widetilde{S}(\cdot), \widetilde{N}(\cdot)$ replacing $Q(\cdot), R(\cdot), S(\cdot), N(\cdot)$ if $\lambda_0$ is sufficiently large. An interior point for the original problem can thus be found by minimizing $\lambda_0$ as the objective for the modified problem. If the minimum of this auxiliary problem is larger than zero, then the original problem is infeasible.
\end{remark}

\section{Convexity and equivalence result of the reformulation}
\label{sec:4}

Algorithm \ref{alg:1} relies on the equivalence of \eqref{eq:coreProblem} and \eqref{eq:alternativeProblem}, and the fact that \eqref{eq:v_t} is a convex barrier function. We prove this fact in this section. 

\begin{lemma}
    \label{lem:RiccatiFeasibility}
    Suppose $R \prec 0$. Then $P\in \bbS^n$ with $F(P,\lambda) \prec 0$ exists if and only if $P_+$ and $P_-$ exist. If this is the case, then the following facts are true:
    \begin{enumerate}
        \item $\forall P \in \bbS^n: F(P,\lambda) \preceq 0 \Rightarrow P_- \preceq P \preceq P_+$,
        \item $\forall \varepsilon > 0 \exists P \in \bbS^n : F(P,\lambda) \prec 0$ and $ P_+ - \varepsilon I \prec P \prec P_+$,
        \item $\Delta \succ 0$.
    \end{enumerate}
\end{lemma}
\begin{proof}
    According to \cite{scherer1991solution}, $R \prec 0$ implies that $P\in \bbS^n$ with $F(P,\lambda) \prec 0$ exists if and only if $P_+$ and $P_-$ exist.
    
    1)~ Due to $R \prec 0$, this fact can be found in \cite{scherer1991solution}.
    
    2)~ Since $P_+$ exists, $(A-B K_+)$ is anti-stable, where $K_+ = R^{-1}(P_+B + S)^\top$. Therefore, the Lyapunov equation
    \begin{align*}
        (A-BK_+)^\top H + H (A-BK_+) = I
    \end{align*}
    has a solution $H \succ 0$. Now, for $\varepsilon > 0$, consider
    \begin{align*}
        F(P_+-\varepsilon H, \lambda )
        &=
        A^\top (P_+- \varepsilon H) + (P_+-\varepsilon H)A + Q\\
        &- (S + (P_+ - \varepsilon H) B) R^{-1} (S + (P_+ - \varepsilon H)B)^\top\\
        &= -\varepsilon (A-BK_+)^\top H - \varepsilon H (A-BK_+)\\
        &- \varepsilon^2 H B R^{-1} B^\top H\\
        &= -\varepsilon I - \varepsilon^2 H B R^{-1} B^\top H.
    \end{align*}
    This expansion proves that there exists an $\varepsilon_0 > 0$, such that $P := P_+ - \varepsilon H$ is feasible for \eqref{eq:LMI} for all $\varepsilon \in ]0,\varepsilon_0[$.
    
    3)~ Due to 2), there exists $P\in \bbS^n$ with $P\prec P_+$. Hence, by 1), we have
    $P_- \preceq P \prec P_+$.
\end{proof}

Next, we can show the equivalence of \eqref{eq:coreProblem} and \eqref{eq:alternativeProblem}.

\emph{Proof of Theorem \ref{thm:equivalence}:} 
Let $(\lambda,P)$ be any feasible point of \eqref{eq:coreProblem}. Due to Lemma \ref{lem:KYP} this implies $R \prec 0$ and the existence of $P_+$, i.e., $\lambda \in \calD$. Furthermore, $P \preceq P_+$ holds (Lemma \ref{lem:RiccatiFeasibility}) implying $P_+ \succ 0$ and $c^\top \lambda - \trace \Sigma P \leq c^\top \lambda - \trace \Sigma P_+$. Hence, $\lambda$ is feasible for \eqref{eq:alternativeProblem} and the optimal value of \eqref{eq:alternativeProblem} is smaller than or equal to the value of \eqref{eq:coreProblem}.

Now let $\lambda \in \bbR^p$ be any feasible point of \eqref{eq:alternativeProblem}. Then $R \prec 0$ holds true and the anti-stabilizing solution $P_+$ exists implying (Lemma \ref{lem:RiccatiFeasibility}) the existence of $P \in \bbS^n$ with $P_+ - \varepsilon I \preceq P \prec P_+$ and $F(P,\lambda) \prec 0$ for any $\varepsilon > 0$. Hence, we can choose $\varepsilon$ so small that $P \succ 0$ is guaranteed and we can perform a Schur complement showing that $(P,\lambda)$ also satisfies \eqref{eq:LMI}. Furthermore, $P$ can be moved arbitrarily close to $P_+$ implying that the optimal value of \eqref{eq:coreProblem} is smaller than or equal to the value of \eqref{eq:alternativeProblem}. {\hfill $\blacksquare$}

Theorem \ref{thm:equivalence} already implies that the feasible set of \eqref{eq:alternativeProblem} is convex since it is the projection of the convex feasible set of \eqref{eq:coreProblem} onto the $\lambda$ variable. However, we are also able to show that all the constraint functions and the objective function of \eqref{eq:alternativeProblem} are convex. To this end, we show the convexity (concavity) of the Hermitian valued functions $\lambda \mapsto P_-(\lambda)$ and $\lambda \mapsto P_+(\lambda)$. Such functions are called convex with respect to the cone of positive semi-definite matrices, if $P_-(\alpha \lambda_1 + (1-\alpha) \lambda_2) \preceq \alpha P_-(\lambda_1) + (1-\alpha) P_-(\lambda_2)$ holds for all $\alpha \in [0,1]$ or concave, if $P_+(\alpha \lambda_1 + (1-\alpha) \lambda_2) \succeq \alpha P_+(\lambda_1) + (1-\alpha) P_+(\lambda_2)$ holds for all $\alpha \in [0,1]$ (\cite{boyd2004convex} p. 109). 

\begin{lemma}
   The mapping $\calD \to \bbS^n, \lambda \mapsto P_-(\lambda)$ is convex and the mapping $\calD \to \bbS^n,\lambda \mapsto P_+(\lambda)$ is concave. Furthermore, the mappings $\lambda \mapsto -\log\det P_+ (\lambda)$, $\lambda \mapsto -\log\det \Delta(\lambda)$, and $\lambda \mapsto - \trace \Sigma P_+(\lambda)$ are convex.
\end{lemma}

\begin{proof}
    W.l.o.g. consider $\lambda \mapsto P_+ (\lambda)$ and two arbitrary $\lambda_1,\lambda_2 \in \calD$. Then $P_1 = P_+(\lambda_1)$ and $P_2 = P_+(\lambda_2)$ are solutions of the Riccati equation and thus both satisfy the constraint \eqref{eq:LMI}. Since \eqref{eq:LMI} is a convex constraint in both $\lambda$ and $P$, also $\lambda_\alpha = \alpha \lambda_1 + (1-\alpha) \lambda_2$ and $P_\alpha = \alpha P_1 + (1-\alpha) P_2)$ satisfy \eqref{eq:LMI} for any $\alpha \in [0,1]$. Consequently, $P_\alpha$ satisfies the Riccati inequality $F(P_\alpha,\lambda) \preceq 0$ for $\lambda = \lambda_\alpha$ implying by 1) of Lemma \ref{lem:RiccatiFeasibility} that
    \begin{align*}
        \alpha P_+(\lambda_1) + (1-\alpha) P_+(\lambda_2) = P_\alpha
        \preceq P_+(\alpha \lambda_1 + (1-\alpha) \lambda_2)
    \end{align*}
    holds. This shows the concavity of $P_+(\cdot)$.

    The convexity of the $\log \det$ functions and the cost function of \eqref{eq:alternativeProblem} follows from the composition theorem (\cite{boyd2004convex} page 110) for convex functions.
\end{proof}

A key role in our barrier function \eqref{eq:v_t} is played by the difference $\Delta$ between the stabilizing and anti-stabilizing solution of the Riccati equation. This difference can be obtained by solving the Riccati equation twice or, more efficiently, it can be obtained from only one solution $P_+$ of the Riccati equation and then solving a Lyapunov equation, according to the following lemma.

\begin{lemma}
    \label{lem:secondRiccatiSol}
    Let $P_1, P_2$ denote two solutions of the Riccati equation $F(P,\lambda) = 0$ and $K_{1} = R^{-1}(S + P_{1}B)^\top$ the controller gain of $P_1$. If the difference $Y = P_2 - P_1$ is invertible, then it satisfies the Lyapunov equation
    \begin{align}
        Y^{-1}(A-BK_1)^\top + (A-BK_1) Y^{-1} &= BR^{-1}B^\top \label{eq:lyapDeltaP1}
    \end{align}
    and $(A-BK_1)$ has no eigenvalues on the imaginary axis.
\end{lemma}

\begin{proof}
    Since $P_2$ is a solution of the Riccati equation, $F(P_2,\lambda) = 0$ holds true. Substituting $P_1 + Y$ for $P_2$ yields
    \begin{align*}
        0 &=  A^\top (P_1 + Y) + (P_1 + Y) A + Q\\
        &- (S + (P_1 + Y)B) R^{-1} (S + (P_1 + Y) B)^\top .
    \end{align*}
    By rearranging terms and using $F(P_1,Y) = 0$ we obtain
    \begin{align*}
        0 &= A^\top Y + Y A - Y B K_1 - K_1^\top B^\top Y - Y B R^{-1} B^\top Y.
    \end{align*}
    Multiplying this equation from both sides by $Y$ yields the claimed Lyapunov equation.
    Next, we show that $(A-BK_1)$ has no imaginary eigenvalues. To this end, assume that $w$ is an eigenvector of $(A-BK_1)^\top$ with imaginary eigenvalue $\mu$. Multiplying \eqref{eq:lyapDeltaP1} from both sides by $w$ yields
    \begin{align*}
        w^* BR^{-1} B^\top w = w^* Y^{-1} (\mu w) + (\mu w)^* Y^{-1} w = 0.
    \end{align*}
    Since $BR^{-1} B^\top$ is negative semi-definite, this implies $BR^{-1} B^\top w = 0$. The latter cannot be true, since $(A,B)$ is controllable implying that $((A-BK_1)^\top,BR^{-1} B^\top)$ is observable. This contradicts the existence of an eigenvector of $(A-BK_1)^\top$ with $BR^{-1} B^\top w = 0$ by the Hautus Lemma.
\end{proof}

Both the Riccati equation \eqref{eq:Riccati} and the Lyapunov equation \eqref{eq:lyapDeltaP1} also appear in \cite{KaoMeg07}. There, these equations are solved to obtain the factorization of a transfer matrix involved in their barrier function. 
Our arguments show that solving this Riccati and Lyapunov equation corresponds to computing both solutions of the Riccati equation \eqref{eq:Riccati}.

Finally, we conclude in the following lemma that \eqref{eq:v_t} is indeed a suitable barrier function for the problem \eqref{eq:alternativeProblem}.

\begin{theorem}
    Let $t > 0$ be fixed and let $(\lambda_k)$ be a convergent sequence of feasible values for \eqref{eq:alternativeProblem}. If the limit $\bar{\lambda}$ of $(\lambda_k)$ lies at the boundary of the feasible set of \eqref{eq:alternativeProblem}, then $v_t(\lambda_k)$ converges to infinity.
\end{theorem}

\begin{proof}
    Since $\bar{\lambda}$ is on the boundary of the feasible set of \eqref{eq:alternativeProblem}, we can perturb the problem data $Q,R,N$ as in Remark \ref{rem:initalFeasible} to $\widetilde{Q}(\lambda) = Q(\lambda) - \lambda_0 I$, $\widetilde{R}(\lambda) = R(\lambda) - \lambda_0 I$ and $\widetilde{N}(\lambda) = N(\lambda) + \lambda_0 I$ with $\lambda_0 > 0$. For the perturbed problem, all $\lambda_k$ and $\bar{\lambda}$ are feasible and $P_+$ and $P_-$ satisfy the strict Riccati inequality. Consequently, $P_+$ and $P_-$ satisfy by Lemma \ref{lem:RiccatiFeasibility} the inequality
    \begin{align*}
        \widetilde{P}_-(\lambda_k) \preceq P_-(\lambda_k) \preceq P_+(\lambda_k) \preceq \widetilde{P}_+(\lambda_k)
    \end{align*}
    for all $k \in \bbN$, where $\widetilde{P}_-(\lambda_k)$ and $\widetilde{P}_+(\lambda_k)$ are the solutions of the perturbed Riccati equations. Since $\widetilde{P}_-(\lambda_k)$ and $\widetilde{P}_+(\lambda_k)$ converge to $\widetilde{P}_-(\bar{\lambda})$ and $\widetilde{P}_+(\bar{\lambda})$, the sequences $P_+(\lambda_k)$ and $P_-(\lambda_k)$ are bounded and, consequently, all log-determinants in \eqref{eq:v_t} are bounded from below.

    If $R(\bar{\lambda}) \not \prec 0$, then, by continuity, we infer $R(\bar{\lambda}) \preceq 0$ and $\det R(\bar{\lambda}) = 0$ which implies that one of the terms in \eqref{eq:v_t} goes to infinity while the others are bounded from below.

    Hence, suppose $R(\bar{\lambda}) \prec 0$. If $\det \Delta (\lambda_k)\to 0$, then $v_t(\lambda_k)$ also goes to infinity. If $\det \Delta (\lambda_k)$ does not converge to zero, then there exist accumulation points $\ovl{P}_-$ and $\ovl{P}_+$ of $P_-(\lambda_k)$ and $P_+(\lambda_k)$ with $\det (\ovl{P}_+ - \ovl{P}_-) \neq 0$, since $P_-(\lambda_k)$ and $P_+(\lambda_k)$ are bounded sequences. By continuity we infer $\ovl{P}_+ - \ovl{P}_- \succ 0$ and that $\ovl{P}_+$ and $\ovl{P}_-$ solve the Riccati equation. Hence, by Lemma \ref{lem:secondRiccatiSol}, the eigenvalues of $(A-B\ovl{K}_+)$ and $(A-B\ovl{K}_-)$ cannot lie on the imaginary axis implying that $\ovl{P}_-$ and $\ovl{P}_+$ are (anti-) stabilizing solutions of $F(P,\bar{\lambda}) = 0$. In this case, we infer $\bar{\lambda} \in \calD$ implying that $\bar{\lambda}$ can only be infeasible if $P_+(\bar{\lambda}) \not \succ 0$. Then, however, we also obtain $v_t(\lambda_k) \to \infty$.
\end{proof}

\section{On the complexity of Algorithm \ref{alg:1}}
\label{sec:5}

The complexity of one (Newton) iteration of Algorithm~\ref{alg:1} is dominated by evaluating $P_+$ and $P_-$ and by computing the derivatives $\partial_{\lambda_i} P_+$, $\partial_{\lambda_i} P_-$ and $\partial_{\lambda_i} \partial_{\lambda_j} P_+, \partial_{\lambda_i} \partial_{\lambda_j} P_-$ of these matrices for $i , j = 1,\ldots, p$. To this end, $q_1$ Riccati equations need to be solved, where $q_1$ is the number of line search iterations, and $q_1 + p(p+3)$ Lyapunov equations need to be solved. Here, two times $p$ Lyapunov equations are required for the first order derivatives and two times $p(p+1)/2$ Lyapunov equations are required for the second order derivatives of $P_+$ and $P_-$. Since the matrix variables involved in these Riccati and Lyapunov equations are of the size $n\times n$, we can refer to \cite{ramesh1989computational} for complexity results. Using e.g. the Schur method, the leading term of the multiplication/division operations required for solving the Riccati equation is $45q_2 n^3$, where $q_2$ is the \emph{average number of double QR-iterations required to make a sub diagonal element equal to zero} \cite{ramesh1989computational}. For Lyapunov equations, the leading term of the complexity bound can be reduced to $5q_2 n^3$. 

Summing this up leads to an asymptotic complexity of $5n^3q_2 (10 q_1 + p(p+3))$ for solving Riccait and Lyapunov equations. In addition, there is a computational effort of $O(q_1p(n^2+m^2+r^2)$ for evaluating $Q(\lambda), S(\lambda), R(\lambda)$ and $N(\lambda)$, of $O(p^2(n + m + r)^2 + p(n + m + r)^3)$ for evaluating the log-determinant and the formulas for its derivatives, of $O(p^2(nm^2 + n^2))$ for setting up the Riccati and Lyapunov equations, and of $O(p^3)$ for solving the Newton system. However, these should all be dominated by the complexity of solving Riccati and Lyapunov equations.

\section{Application example: Robust state feedback design}
\label{sec:6}

Unlike the prior works we cited in the introduction, Algorithm \ref{alg:1} enables the solution of KYP-LMIs for state-feedback synthesis. Thus, we consider as a benchmark a robust LQR synthesis task for dynamical systems
\begin{align}
	\begin{pmatrix}
		\dot{x}(t)\\
		z(t)
	\end{pmatrix}
	=
	\begin{pmatrix}
		\calA & \calB_1 & \calB_2\\
		\calC & \calD_1 & \calD_2
	\end{pmatrix}
	\begin{pmatrix}
		x(t)\\
		u(t)\\
		w(t)
	\end{pmatrix}. \label{eq:cont_sys}
\end{align}
In this state space description, $x(t) \in \bbR^n$ is the state, $u(t) \in \bbR^m$ is the control input, and $w(t) \in \bbR^{d}$ and $z(t) \in \bbR^{l}$ are the input and output of an uncertain system component. We assume that this uncertain component satisfies for all times the family of quadratic constraints
\begin{align}
	\begin{pmatrix}
		\calC x + \calD_1 u + \calD_2 w\\
		w
	\end{pmatrix}^\top M(\lambda)^{-1}
	\begin{pmatrix}
		\calC x + \calD_1 u + \calD_2 w\\
		w
	\end{pmatrix} &\geq 0 \label{eq:multiplier}
\end{align}
for all $\lambda \in \bbR^p$ with $N(\lambda) \succ 0$. Our goal is finding a robust performance control Lyapunov function $V: \bbR^n \to \bbR_{\geq 0}, x \mapsto x^\top P^{-1} x$ with positive definite $P = P^\top$ such that
\begin{align}
	\min_{u\in \bbR^m}\nabla V(x)^\top (\calA x + \calB_1 u + \calB_2 w) + x^\top \calQ x + u^\top \calR u \leq 0 \label{eq:robust_LQR}
\end{align}
holds true for all $x\in \bbR^n$ and all $w \in \bbR^d$ satisfying \eqref{eq:multiplier}. In \eqref{eq:robust_LQR}, $\calQ$ and $\calR$ are positive definite matrices and $x^\top \calQ x + u^\top \calR u$ is a stage cost function. As we show in Appendix \ref{app:1}, using standard  techniques from robust control, such a Lyapunov function can be found by solving the SDP
\begin{align}
	\minimize_{P \in \bbS^n,\lambda \in \bbR^p} ~&~ -\trace P \label{eq:dual_opt}
\end{align}
subject to $P \succ 0$, $N(\lambda) \succ 0$ and
\begin{align}
	&\resizebox{0.89\linewidth}{!}{$0 \succ \left(
	\begin{array}{c|cc}
		\calA^\top & I & \calC^\top\\ \hline
		I & 0 & 0
	\end{array}
	\right)^\top
	\left(
	\begin{array}{c|c}
		0 & P\\\hline
		P & 0
	\end{array}
	\right)
	\left(
	\begin{array}{c|cc}
		\calA^\top & I & \calC^\top\\ \hline
		I & 0 & 0
	\end{array}
	\right)-$} \label{eq:robustSynthesisKYP}\\
	&(\star)^\top \begin{pmatrix}
		\calQ^{-1} & 0 & &\\
		0 & \calR^{-1} & &\\
		& & M_{11}(\lambda) & M_{12}(\lambda)\\
		& & M_{21}(\lambda) & M_{22}(\lambda)
	\end{pmatrix}
	\left(
	\begin{array}{c|cc}
		0 & -I & 0\\
		\calB_1^\top & 0 & \calD_1^\top\\
		0 & 0 & -I\\
		\calB_2^\top & 0 & \calD_2^\top
	\end{array}\right) \nonumber
\end{align}
if the family of multipliers satisfies the conditions
\begin{align}
	\begin{pmatrix}
		I\\
		\calD_2^\top
	\end{pmatrix}^\top
	\begin{pmatrix}
		M_{11}(\lambda) & M_{12}(\lambda)\\
		M_{21}(\lambda) & M_{22}(\lambda)
	\end{pmatrix}
	\begin{pmatrix}
		I\\
		\calD_2^\top
	\end{pmatrix} &\succ 0,~ M_{22}(\lambda) \prec 0 \label{eq:MultConstraint}
\end{align}
for all $\lambda \in \bbR^p$ with $N(\lambda) \succ 0$.

In order to consider realistic control systems, we use the database \cite{leibfritz2003description} to select the system matrices $\calA$ and $\calB_1$. To model uncertainty (which is not available in \cite{leibfritz2003description}), we assume that the actuators of our controller are subject to a parametric multiplicative uncertainty of $25\%$. This model assumption can be implemented by choosing the matrices $\calC = 0$, $\calD_1 = I$, $\calD_2 = 0$ and $\calB_2 = \calB_1$. Furthermore, the multiplier matrix can be chosen as
\begin{align*}
	M(\lambda) = \diag (\gamma^2\lambda_1,\ldots,\gamma^2\lambda_p,-\lambda_1,\ldots,-\lambda_p),
\end{align*}
where $\gamma = 0.25$. For these system matrices and multiplier matrix, we solve the KYP-SDP \eqref{eq:dual_opt} using Algorithm \ref{alg:1} and the off-the-shelve SDP solvers LMILab, SeDuMi \cite{sturm1999using} and Mosek \cite{mosek}. Solution times for multiple discretizations of an Euler Bernoulli Beam (EB) system, a heat flow (HF) system, and a cable mass (CM) model can be found in Table \ref{tab:computationTimes}. Our implementation, as well as the statistics for all the other models featured in \cite{leibfritz2003description}, are provided on github (\url{https://github.com/SphinxDG/KYP-SDP}).

\begin{table}
	\caption{Computation times for the four solvers. The number of system states is $n$ the number of multipliers (control inputs) is $s = m$. ``-'' indicates that the solver was unable to solve the problem within $10^4s$.}
	\label{tab:computationTimes}
	\centering
	\begin{tabular}{@{}llllllll@{}}
		\toprule
		Problem & n & p & LMILab & SeDuMi & Mosek & Algo \ref{alg:1}
		\\\midrule
		EB1 & 10 & 1 & 0.918s & 0.181s & 0.175s & 0.0189s\\
		EB2 & 10 & 1 & 0.912s & 0.177s & 0.164s & 0.0169s\\
		EB3 & 10 & 1 & 0.969s & 8.11s & 0.177s & 0.0161s\\
		EB4 & 20 & 1 & 121s & 0.582s & 0.202s & 0.0568s\\
		EB5 & 40 & 1 & - & 5.55s & 1.35s & 0.450s\\
		EB6 & 160 & 1 & - & - & 610s & 8.13s\\
		HF2D3 & 4489 & 2 & - & - & - & 8579s\\
		HF2D4 & 2025 & 2 & - & - & - & 715s\\
		HF2D5 & 4489 & 2 & - & - & - & 8670s\\
		HF2D6 & 2025 & 2 & - & - & - & 690s\\
		CM1 & 20 & 1 & 140s & 0.604s & 0.314s & 0.130s\\
		CM2 & 60 & 1 & - & 41.3s & 10.5s & 0.90s\\
		CM3 & 120 & 1 & - & 2770s & 234s & 2.59s\\
		CM4 & 240 & 1 & - & - & 2095s & 19.9s\\
		CM5 & 480 & 1 & - & - & - & 92.0s\\
		CM6 & 960 & 1 & - & - & - & 404s
		\\\bottomrule
	\end{tabular}
\end{table}

\section{Conclusion}
\label{sec:7}

We present a new solver for KYP-SDPs. To exploit the structure of these LMI optimization problems, we formulate an equivalent problem, where the Lyapunov matrix of the KYP-LMI is eliminated by solving a Riccati equation instead. This step removes $O(n^2)$ variables from the SDP and preserves the convexity of the original problem. As we see in Table \ref{tab:computationTimes} the resulting algorithm achieves a significant speed-up compared to off-the-shelve solvers and is able to solve larger problems.

\bibliographystyle{abbrv}
\bibliography{sources.bib}

\appendix

\subsection{Elimination for robust LQR-synthesis (standard)}
\label{app:1}

Denote by $\widetilde{M}$ the matrix
\begin{align*}
	\begin{pmatrix}
		\widetilde{M}_{11} & \widetilde{M}_{12}\\
		\widetilde{M}_{21} & \widetilde{M}_{22}
	\end{pmatrix}
	:=
	\begin{pmatrix}
		M_{11}(\lambda) & M_{12}(\lambda)\\
		M_{21}(\lambda) & M_{22}(\lambda)
	\end{pmatrix}^{-1}.
\end{align*}
The first step to derive the KYP-LMI \eqref{eq:robustSynthesisKYP} is a multiplier relaxation of the constraint \eqref{eq:robust_LQR}. Namely, if there exists a $\lambda \in \bbR^p$ with $N(\lambda) \succ 0$ and a $\calK \in \bbR^{m \times n}$, such that
\begin{align}
	\nabla V(x)^\top ((\calA + \calB_1 \calK)x +& \calB_2 w) + \begin{pmatrix}
		z\\
		w
	\end{pmatrix}^\top \widetilde{M}(\lambda)
	\begin{pmatrix}
		z\\
		w
	\end{pmatrix} \nonumber\\
	&+
	x^\top (\calQ + \calK^\top \calR \calK) x \label{eq:continuous_time_constraint}
\end{align}
is non-positive for all $x\in \bbR^n \setminus \{0\}$ and $w \in \bbR^{d_1}$, then this implies \eqref{eq:robust_LQR}. The new constraint \eqref{eq:continuous_time_constraint} can be denoted as the semi-definite constraint that
\begin{align}
	&(\star)^\top \begin{pmatrix}
		0 & P^{-1} & & &\\
		P^{-1} & 0 & & &\\
		& & \calQ & 0 &\\
		& & 0 & \calR &\\
		&  & & & \widetilde{M}_{11} & \widetilde{M}_{12}\\
		& & & & \widetilde{M}_{21} & \widetilde{M}_{22}
	\end{pmatrix}
	\begin{pmatrix}
		I & 0\\
		\calA + \calB_1 \calK & \calB_2\\
		I & 0\\
		\calK & 0\\
		\calC + \calD_1 \calK & \calD_2\\
		0 & I
	\end{pmatrix}
	\label{eq:primalSynthesis}
\end{align}
must be negative definite. This constraint is non-convex due to $\calK$. Hence, we apply the following elimination lemma (\cite{HELMERSSON19993361}, Theorem 2) to eliminate the variable $\calK$.
\begin{lemma}
        \label{lem:elimination}
	Consider the matrix inequality
	\begin{align}
		\begin{pmatrix}
			I_k\\
			U^\top \calK V + W
		\end{pmatrix}^\top \calP \begin{pmatrix}
		I_k\\
		U^\top \calK V + W
	\end{pmatrix} \prec 0
	\label{eq:QI}
	\end{align}
	and assume that $\calP = \calP^\top$ is invertible with exactly $k$ negative eigenvalues. Let $U_{\perp}, V_{\perp}$ be basis matrices of $\ker (U), \ker (V)$. Then there exists a $\calK \in \bbR^{m \times n}$ such that \eqref{eq:QI} is satisfied if and only if
	\begin{align*}
		V_{\perp}^\top \begin{pmatrix}
			I\\ W
		\end{pmatrix}^\top \calP \begin{pmatrix}
		I\\ W
	\end{pmatrix} V_{\perp} \prec 0 ~\&~ 
	U_{\perp}^\top \begin{pmatrix}
		W^\top \\ -I
	\end{pmatrix}^\top \calP^{-1} \begin{pmatrix}
		W^\top\\ -I
	\end{pmatrix} U_{\perp} \succ 0.
	\end{align*}
\end{lemma}
Note that the assumption on the eigenvalues of the central matrix in \eqref{eq:primalSynthesis} is satisfied, since
\begin{align*}
	\begin{pmatrix}
		0 & P^{-1}\\
		P^{-1} & 0
	\end{pmatrix}
\end{align*}
has $n$ positive and $n$ negative eigenvalues, the matrix $\widetilde{M}$
has $d_1$ negative and $d_2$ positive eigenvalues due to \eqref{eq:MultConstraint} and $\calQ$ and $\calR$ have $n$ and $m$ positive eigenvalues. This makes a total number of $n + d_1$ negative eigenvalues.
Next, we reorder terms in \eqref{eq:primalSynthesis} to bring it to the form \eqref{eq:QI} and enable the application of Lemma \ref{lem:elimination}. This yields
\begin{align*}
	(\star)^\top \begin{pmatrix}
		0 & & P^{-1} & &\\
		& \widetilde{M}_{22} & & & & \widetilde{M}_{21} \\
		P^{-1} & & 0 & &\\
		& & & \calQ & 0 &\\
		& & & 0 & \calR &\\
		& \widetilde{M}_{12} & & & & \widetilde{M}_{11}\\
	\end{pmatrix}
	\begin{pmatrix}
		I_n & 0\\
		0 & I_{d_1}\\
		\calA + \calB_1 \calK & \calB_2\\
		I & 0\\
		\calK & 0\\
		\calC + \calD_1 \calK & \calD_2
	\end{pmatrix}.
\end{align*}
Here, we see that we can choose $\calP$ as the inner matrix of this product and
\begin{align*}
	U^\top \calK V + W = \begin{pmatrix}
		\calB_1\\
		0\\
		I\\
		\calD_1
	\end{pmatrix}\calK \begin{pmatrix}
		I & 0
	\end{pmatrix}
	+
	\begin{pmatrix}
		\calA & \calB_2\\
		I & 0\\
		0 & 0\\
		\calC & \calD_2
	\end{pmatrix}.
\end{align*}
The basis matrices of the kernels can be chosen as
\begin{align*}
	U_\perp &= \begin{pmatrix}
		I & 0 & 0\\
		0 & I & 0\\
		-\calB_1 & 0 & -\calD_1\\
		0 & 0 & I
	\end{pmatrix}, &
	V_{\perp} &= \begin{pmatrix}
		0\\
		I
	\end{pmatrix}.
\end{align*}
Next, by computing the products 
\begin{align*}
	\begin{pmatrix}
		W^\top\\ -I
	\end{pmatrix} U_{\perp}
	&=
	\begin{pmatrix}
		\calA^\top & I & \calC^\top\\
		\calB_2^\top & 0 & \calD_2^\top\\
		-I & 0 & 0\\
		0 & -I & 0\\
		\calB_1 & 0 & \calD_1\\
		0 & 0 & -I
	\end{pmatrix}, &
    \begin{pmatrix}
		I\\ W
	\end{pmatrix} V_{\perp}
    & =
    \begin{pmatrix}
		0\\
		I_{d_1}\\
		\calB_2\\
		0\\
		0\\
		\calD_2
    \end{pmatrix}
\end{align*}
and applying Lemma \ref{lem:elimination}, we can see that \eqref{eq:primalSynthesis} is negative definite if and only if
\begin{align}
	(\star)^\top
	\begin{pmatrix}
		0 & & P^{-1} & &\\
		& \widetilde{M}_{22} & & & & \widetilde{M}_{21} \\
		P^{-1} & & 0 & &\\
		& & & \calQ & 0 &\\
		& & & 0 & \calR &\\
		& \widetilde{M}_{12} & & & & \widetilde{M}_{11}\\
	\end{pmatrix}
	\begin{pmatrix}
		0\\
		I_{d_1}\\
		\calB_2\\
		0\\
		0\\
		\calD_2
	\end{pmatrix}
	\label{eq:largePrimal}
\end{align}
is negative definite and
\begin{align}
	(\star)^\top
	\begin{pmatrix}
		0 & & P & &\\
		& M_{22} & & & & M_{21} \\
		P & & 0 & &\\
		& & & \calQ^{-1} & 0 &\\
		& & & 0 & \calR^{-1} &\\
		& M_{12} & & & & M_{11}\\
	\end{pmatrix}
	\begin{pmatrix}
		\calA^\top & I & \calC^\top\\
		\calB_2^\top & 0 & \calD_2^\top\\
		-I & 0 & 0\\
		0 & -I & 0\\
		\calB_1 & 0 & \calD_1\\
		0 & 0 & -I
	\end{pmatrix}
	\label{eq:badlySortedDual}
\end{align}
is positive definite. Rearranging terms again in \eqref{eq:badlySortedDual} yields \eqref{eq:robustSynthesisKYP} and multiplying out \eqref{eq:largePrimal} shows that this constraint is included in \eqref{eq:MultConstraint}.

\end{document}